\documentclass[11pt]{article}

\usepackage[margin=1in]{geometry}
\usepackage[T1]{fontenc}
\usepackage[utf8]{inputenc}
\usepackage{lmodern}
\usepackage{amsmath,amssymb,amsfonts,amsthm,mathtools}
\usepackage{mathrsfs}
\usepackage{booktabs,longtable,array}
\usepackage{enumitem}
\usepackage{microtype}
\usepackage[colorlinks=true,linkcolor=blue,citecolor=blue,urlcolor=blue]{hyperref}
\usepackage[capitalize,nameinlink]{cleveref}

\allowdisplaybreaks
\numberwithin{equation}{section}

\newtheorem{theorem}{Theorem}[section]
\newtheorem{proposition}[theorem]{Proposition}
\newtheorem{lemma}[theorem]{Lemma}
\newtheorem{corollary}[theorem]{Corollary}
\newtheorem{definition}[theorem]{Definition}
\newtheorem{remark}[theorem]{Remark}
\newtheorem{example}[theorem]{Example}

\crefname{theorem}{Theorem}{Theorems}
\Crefname{theorem}{Theorem}{Theorems}
\crefname{proposition}{Proposition}{Propositions}
\Crefname{proposition}{Proposition}{Propositions}
\crefname{lemma}{Lemma}{Lemmas}
\Crefname{lemma}{Lemma}{Lemmas}
\crefname{corollary}{Corollary}{Corollaries}
\Crefname{corollary}{Corollary}{Corollaries}
\crefname{definition}{Definition}{Definitions}
\Crefname{definition}{Definition}{Definitions}
\crefname{remark}{Remark}{Remarks}
\Crefname{remark}{Remark}{Remarks}
\crefname{example}{Example}{Examples}
\Crefname{example}{Example}{Examples}

\DeclareMathOperator{\supp}{supp}

\newcommand{\R}{\mathbb R}

\newcommand{\eps}{\varepsilon}
\newcommand{\loc}{\mathrm{loc}}

\newcommand{\dd}{\,\mathrm d}
\newcommand{\dx}{\,\mathrm dx}
\newcommand{\dt}{\,\mathrm dt}
\newcommand{\dxdt}{\,\mathrm dx\,\mathrm dt}
\newcommand{\dxds}{\,\mathrm dx\,\mathrm ds}
\newcommand{\dy}{\,\mathrm dy}

\newcommand{\calK}{\mathcal K}

\newcommand{\norm}[2]{\left\lVert #1\right\rVert_{#2}}
\newcommand{\abs}[1]{\left\lvert #1\right\rvert}

\newcommand{\pair}[2]{\left\langle #1,#2\right\rangle}
\newcommand{\Pf}{\mathsf P}

\newcommand{\one}{\mathbf 1}

\newcommand{\W}{\mathcal W}
\newcommand{\D}{\mathcal D}
\newcommand{\Lk}{\mathcal L}
\newcommand{\Ek}{\mathcal E}
\newcommand{\Aact}{\mathfrak A}
\newcommand{\Bback}{\mathcal B}
\newcommand{\Cpf}{\mathcal C^{\mathrm{PF}}}

\title{\textbf{Coarse-Grained Resolution and Pressure--Flux Work Depletion for Navier--Stokes CKN Badness}}

\author{Runlong Yu\\
The University of Alabama,
Tuscaloosa, AL 35487, USA\\
\texttt{ryu5@ua.edu}}
\date{}

\begin{document}
\maketitle

\begin{abstract}
We prove a finite-scale coarse-grained decomposition for the three-dimensional incompressible Navier--Stokes equations near the Caffarelli--Kohn--Nirenberg local regularity framework.  The first part is a local resolution lemma for the scale-critical quantity: for every spatial filter length \(\ell>0\),
\[
\Psi(r)\le 4\Psi^\ell(r)+4\Omega^\ell(r),
\]
where \(\Psi^\ell(r)\) is the corresponding coarse-grained velocity--pressure quantity and \(\Omega^\ell(r)\) is the explicitly defined subfilter residual.  Thus a CKN-bad scale is either visible at the resolved level or is carried by unresolved velocity--pressure oscillation.  The second part is an exact fixed-chain depletion theorem for the combined pressure--flux work distribution
\[
G^\ell=\Pi^\ell+\nabla\cdot(P^\ell U^\ell),
\qquad
\Pi^\ell=-R^\ell:\nabla U^\ell,
\]
which is the signed work density appearing in the localized resolved-energy balance.  For finite-dimensional active test families with common endpoint traces, we obtain a constructive active-work extraction and a weighted telescoping inequality: forward combined work and resolved dissipation are paid by the initial localized kinetic energy, explicit localization leakage, and negative combined work/backscatter.  
\end{abstract}



\tableofcontents

\section{Introduction}\label{sec:introduction}

Let \((u,p)\) be a suitable weak solution of the three-dimensional incompressible Navier--Stokes equations
\begin{equation}\label{eq:NS}
\partial_t u-\Delta u+\nabla\cdot(u\otimes u)+\nabla p=0,
\qquad \nabla\cdot u=0.
\end{equation}
The weak-solution framework goes back to Leray and Hopf \cite{Leray1934,Hopf1951}.  The local energy inequality and the partial regularity theory of Scheffer and Caffarelli--Kohn--Nirenberg place the local regularity problem in terms of scale-critical quantities; see \cite{Scheffer1976,CKN1982,Lin1998,Vasseur2007,LemarieRieusset2002,LemarieRieusset2016,Seregin2015}.  A standard CKN quantity in a parabolic cylinder is
\begin{equation}\label{eq:CKN-intro}
r^{-2}\int_{Q_r(z_0)}\left(|u|^3+|p-(p)_{B_r(x_0)}(t)|^{3/2}\right)\dxdt.
\end{equation}
The subtraction of a spatial pressure mean is essential because the pressure is determined only up to an arbitrary function of time.

A different but compatible viewpoint is obtained by spatial coarse graining.  Smooth filters and resolved energy balances are standard tools in the mathematical study of energy transfer and anomalous dissipation; compare \cite{ConstantinETiti1994,DuchonRobert2000,Eyink2005,EyinkAluie2009,DrivasNguyen2018}.  For a filter length \(\ell>0\), set
\[
U^\ell=S_\ell u,
\qquad
P^\ell=S_\ell p,
\qquad
R^\ell=S_\ell(u\otimes u)-U^\ell\otimes U^\ell.
\]
The resolved interscale flux and the combined pressure--flux work distribution are
\begin{equation}\label{eq:G-intro}
\Pi^\ell=-R^\ell:\nabla U^\ell,
\qquad
G^\ell=\Pi^\ell+\nabla\cdot(P^\ell U^\ell).
\end{equation}
The object \(G^\ell\) is the scalar work distribution that appears in the localized resolved-energy identity.  It is signed; it is not a pressure norm, not an unsigned flux norm, and not automatically comparable to the positive CKN quantity \eqref{eq:CKN-intro}.

The first result of this paper is a resolution lemma for \eqref{eq:CKN-intro}.  With \(\Psi(r)\) denoting the full CKN quantity, \(\Psi^\ell(r)\) the corresponding resolved quantity, and \(\Omega^\ell(r)\) the residual carried by \((u-U^\ell,p-P^\ell)\), we prove
\[
\Psi(r)\le 4\Psi^\ell(r)+4\Omega^\ell(r).
\]
Consequently, a full CKN-bad scale is either visible in the resolved velocity--pressure pair or else the subfilter residual is quantitatively large.  This statement is algebraic, gauge invariant, and independent of the Navier--Stokes equation.  Its role is to prevent a detector based on \((U^\ell,P^\ell,R^\ell)\) from being asked to detect information that the filter has removed.

The second result is a finite-chain work depletion theorem for \(G^\ell\).  On a fixed chain of adjacent slabs, finite-dimensional active test families produce a selected nonnegative weight on each slab.  The selected combined work satisfies an explicit extraction lower bound in terms of a finite coefficient norm.  Summing the exact resolved-energy identity over the chain gives a weighted telescoping estimate: forward combined pressure--flux work and resolved dissipation are controlled by initial localized kinetic energy, localization leakage, and negative combined work, recorded as backscatter.

These two results fit together through the following conditional implication:
\[
\boxed{
\Psi(r_k)\ge \varepsilon_0
\Longrightarrow
\Omega^\ell(r_k)>\eta\varepsilon_0
\quad\text{or}\quad
\Psi^\ell(r_k)\ge c_0\varepsilon_0,
}
\]
and, if an independent coarse observability estimate is available,
\[
\boxed{
\Psi^\ell(r_k)\ge c_0\varepsilon_0
\Longrightarrow
\mathfrak A_k(G^\ell)\ge c_{\mathrm{obs}},
}
\]
then the depletion theorem applies to the detected resolved work.  The difficult implication is the second boxed line.  It is not proved here.  It is a separate observability or compactness-rigidity problem and may fail through residual concentration, harmonic pressure tails, pressure--flux cancellation, coherent low-frequency resolved profiles, leakage, or backscatter.

The local pressure decompositions used below rely on standard Calder\'on--Zygmund and harmonic estimates; see \cite{Stein1970}.  Related finite-window reductions and detector formulations appear in \cite{Yu2026InvisibleCascades,Yu2026CriticalLedgers,Yu2026SingularityAudit,Yu2026ComputationalAntiPhantom,Yu2026SharpLocalToClean}; the present paper isolates the unconditional resolution and fixed-chain work identities needed for that interface.

The paper is organized as follows.  \Cref{sec:resolution-main} proves the coarse/residual CKN resolution lemma, finite-chain form, and residual estimates.  \Cref{sec:work-depletion} proves the fixed-chain pressure--flux work depletion theorem.  \Cref{sec:conditional-interface} records the precise conditional bridge between CKN resolution and work detection.  \Cref{sec:diagnostics} gives examples, scaling information, and limitations.

\section{Preliminaries}

\subsection{Parabolic cylinders and pressure projection}

For $z_0=(x_0,t_0)\in\R^3\times\R$ and $r>0$, set
\[
B_r(x_0)=\{x\in\R^3: |x-x_0|<r\},
\qquad
Q_r(z_0)=B_r(x_0)\times(t_0-r^2,t_0).
\]
When $z_0=(0,0)$, we write $B_r=B_r(0)$ and $Q_r=Q_r(0,0)$.
For an integrable function $f$ on $B_r(x_0)$, define
\[
(f)_{B_r(x_0)}(t)=\frac{1}{|B_r|}\int_{B_r(x_0)} f(x,t)\dx.
\]
We also use the pressure projection
\begin{equation}\label{eq:pressure-projection}
\Pf_{r,z_0} f(x,t)=f(x,t)-(f)_{B_r(x_0)}(t),
\qquad (x,t)\in Q_r(z_0).
\end{equation}
At the origin, we write simply $\Pf_r$.

The map $\Pf_{r,z_0}$ is linear.  In particular, if $p=P+q$, then
\begin{equation}\label{eq:pressure-projection-linearity}
\Pf_{r,z_0}p=\Pf_{r,z_0}P+\Pf_{r,z_0}q.
\end{equation}
This elementary identity is the reason the pressure part of the resolution lemma is gauge invariant.

\subsection{Suitable weak solutions}

\begin{definition}[Suitable weak solution]
Let $Q\subset\R^3\times\R$ be a parabolic cylinder.  A pair $(u,p)$ is a suitable weak solution of \eqref{eq:NS} in $Q$ if
\[
u\in L_t^\infty L_x^2(Q)\cap L_t^2H_x^1(Q),
\qquad
p\in L^{3/2}(Q),
\]
$(u,p)$ solves \eqref{eq:NS} in the sense of distributions, $\nabla\cdot u=0$, and the local energy inequality holds: for every nonnegative $\phi\in C_c^\infty(Q)$ and almost every $t$,
\begin{align}
\int |u(x,t)|^2\phi(x,t)\dx
&+2\int_{-\infty}^{t}\int |\nabla u|^2\phi\dxds
\nonumber\\
&\le
\int_{-\infty}^{t}\int |u|^2(\partial_s\phi+\Delta\phi)\dxds
\nonumber\\
&\quad+
\int_{-\infty}^{t}\int (|u|^2+2p)u\cdot\nabla\phi\dxds.
\end{align}
\end{definition}

The resolution lemma below only needs $u\in L^3$ and $p\in L^{3/2}$ locally.  The suitable weak solution framework is included because it is the natural setting in which $\Psi(r)$ is used for local regularity.

\subsection{Spatial coarse graining}

Let $\rho\in C_c^\infty(B_1)$ be a nonnegative radial mollifier satisfying $\int_{\R^3}\rho\dx=1$.  For $\ell>0$, set
\[
\rho_\ell(x)=\ell^{-3}\rho(x/\ell),
\qquad
S_\ell f=\rho_\ell *_x f.
\]
The convolution is in the spatial variables only.  Whenever $Q_r(z_0)$ is under discussion, the fields are assumed to be defined on a spatially enlarged cylinder large enough for the convolution to be evaluated on $Q_r(z_0)$.

For a velocity-pressure pair $(u,p)$, define
\begin{equation}\label{eq:coarse-fields}
U^\ell=S_\ell u,
\qquad
P^\ell=S_\ell p,
\qquad
R^\ell=S_\ell(u\otimes u)-U^\ell\otimes U^\ell.
\end{equation}
The resolved interscale work and the combined pressure--flux work distribution are
\begin{equation}\label{eq:work-defs}
\Pi^\ell=-R^\ell:\nabla U^\ell,
\qquad
G^\ell=\Pi^\ell+\nabla\cdot(P^\ell U^\ell).
\end{equation}
The first main theorem concerns the CKN badness resolved by $(U^\ell,P^\ell)$; the work distribution $G^\ell$ enters in the finite-chain depletion theorem.

\subsection{Full, coarse, and residual CKN quantities}

Fix a cylinder $Q_r(z_0)\Subset Q$ and a coarse length $\ell>0$.  Define the full CKN badness
\begin{equation}\label{eq:Psi-full}
\Psi(z_0,r)=r^{-2}\int_{Q_r(z_0)}\left(|u|^3+|\Pf_{r,z_0}p|^{3/2}\right)\dxdt.
\end{equation}
Define the coarse badness
\begin{equation}\label{eq:Psi-coarse}
\Psi^\ell(z_0,r)=r^{-2}\int_{Q_r(z_0)}\left(|U^\ell|^3+|\Pf_{r,z_0}P^\ell|^{3/2}\right)\dxdt.
\end{equation}
Define the subfilter residual
\begin{equation}\label{eq:Omega-residual}
\Omega^\ell(z_0,r)=r^{-2}\int_{Q_r(z_0)}\left(|u-U^\ell|^3+|\Pf_{r,z_0}(p-P^\ell)|^{3/2}\right)\dxdt.
\end{equation}
At $z_0=(0,0)$ we omit $z_0$ from the notation.

All three quantities are invariant under the Navier--Stokes scaling if the ratio $\ell/r$ is kept fixed.  The full and residual pressure terms are invariant under the transformation $p\mapsto p+a(t)$, and the same is true for the coarse pressure term because $S_\ell$ acts only in space.
\section{Coarse-Grained CKN Resolution}\label{sec:resolution-main}

\subsection{The resolution lemma}

\begin{theorem}[Resolution lemma]\label{thm:resolution}
Let $(u,p)$ be locally integrable with $u\in L^3_{\loc}$ and $p\in L^{3/2}_{\loc}$, and let $(U^\ell,P^\ell)$ be defined by spatial coarse graining as in \eqref{eq:coarse-fields}.  For every cylinder $Q_r(z_0)$ on which the quantities \eqref{eq:Psi-full}--\eqref{eq:Omega-residual} are finite,
\begin{equation}\label{eq:resolution-estimate}
\Psi(z_0,r)\le 4\Psi^\ell(z_0,r)+4\Omega^\ell(z_0,r).
\end{equation}
More precisely,
\begin{align}\label{eq:sharp-resolution-estimate}
\Psi(z_0,r)
&\le
4r^{-2}\int_{Q_r(z_0)}\left(|U^\ell|^3+|u-U^\ell|^3\right)\dxdt
\nonumber\\
&\quad+
2^{1/2}r^{-2}\int_{Q_r(z_0)}\left(|\Pf_{r,z_0}P^\ell|^{3/2}+|\Pf_{r,z_0}(p-P^\ell)|^{3/2}\right)\dxdt.
\end{align}
\end{theorem}

\begin{corollary}[Coarse visibility below residual threshold]\label{cor:coarse-visibility}
Let $\eps_0>0$ and $0<\eta<1/4$.  If
\[
\Psi(z_0,r)\ge \eps_0
\qquad\text{and}\qquad
\Omega^\ell(z_0,r)\le \eta\eps_0,
\]
then
\begin{equation}\label{eq:coarse-lower-bound}
\Psi^\ell(z_0,r)
\ge
\left(\frac14-\eta\right)\eps_0.
\end{equation}
Equivalently, if $\Psi^\ell(z_0,r)<(1/4-\eta)\eps_0$ while $\Psi(z_0,r)\ge\eps_0$, then $\Omega^\ell(z_0,r)>\eta\eps_0$.
\end{corollary}

\begin{remark}[Minimal hypotheses]
The estimate is independent of the equation.  It remains valid for any $u\in L^3$ and $p\in L^{3/2}$ on the cylinder.  No local pressure formula, no local energy inequality, and no smallness condition on $\ell/r$ is used.
\end{remark}

\subsection{Finite-chain form}

Many local regularity and detector arguments use a finite geometric chain of scales.  Fix $0<\theta<1$ and set
\[
r_k=\theta^k r_0,
\qquad k=0,1,\dots,N.
\]
For a fixed coarse length $\ell$, define the full, coarse, and residual bad-scale sets
\begin{align*}
\mathcal B_{\eps_0}&=\{0\le k\le N-1:\ \Psi(z_0,r_k)\ge\eps_0\},\\
\mathcal B^\ell_{c_0\eps_0}&=\{0\le k\le N-1:\ \Psi^\ell(z_0,r_k)\ge c_0\eps_0\},\\
\mathcal R^\ell_{\eta\eps_0}&=\{0\le k\le N-1:\ \Omega^\ell(z_0,r_k)>\eta\eps_0\}.
\end{align*}

\begin{corollary}[Finite-chain resolution]\label{cor:finite-chain}
Let $0<\eta<1/4$ and set $c_0=1/4-\eta$.  Then
\begin{equation}\label{eq:bad-set-inclusion}
\mathcal B_{\eps_0}\subset
\mathcal B^\ell_{c_0\eps_0}\cup \mathcal R^\ell_{\eta\eps_0}.
\end{equation}
Consequently,
\begin{equation}\label{eq:bad-set-count}
\#\mathcal B_{\eps_0}
\le
\#\mathcal B^\ell_{c_0\eps_0}+
\#\mathcal R^\ell_{\eta\eps_0}.
\end{equation}
More generally, for every nonnegative sequence of weights $w_k$,
\begin{equation}\label{eq:weighted-bad-set}
\sum_{k=0}^{N-1}w_k\one_{\{\Psi(r_k)\ge\eps_0\}}
\le
\sum_{k=0}^{N-1}w_k\one_{\{\Psi^\ell(r_k)\ge c_0\eps_0\}}
+
\sum_{k=0}^{N-1}w_k\one_{\{\Omega^\ell(r_k)>\eta\eps_0\}}.
\end{equation}
\end{corollary}

This is the finite-chain bookkeeping form of the lemma.  It does not count CKN-bad scales by itself; it says that every such count splits into a coarse-visible count and a residual count.

\subsection{Conditional detector interface}

The following proposition records how the resolution lemma would enter a detector theorem once an independent coarse observability estimate is available.

Let $A_k^\infty(G^\ell)$ be any scale-$r_k$ detector seminorm built from the combined work distribution $G^\ell$, for instance
\begin{equation}\label{eq:infinite-detector}
A_k^\infty(G^\ell)=
\sup_{\psi\in\mathfrak X_k,\ \norm{\psi}{\mathfrak X_k}\le1}
 r_k^{-1}\left|\left\langle G^\ell,\chi_k\psi_k^\sharp\right\rangle\right|,
\end{equation}
where $\chi_k$ is a cutoff on $Q_{r_k}$ and $\psi_k^\sharp$ is the parabolically rescaled reference test profile.

\begin{proposition}[Resolution-to-detector template]\label{prop:detector-template}
Fix $k$, $\eps_0>0$, and $0<\eta<1/4$.  Suppose that on a class $\calK_k$ of coarse-grained Navier--Stokes packages one has the coarse observability implication
\begin{equation}\label{eq:coarse-observability-hyp}
\Psi^\ell(r_k)\ge\left(\frac14-\eta\right)\eps_0
\quad\Longrightarrow\quad
A_k^\infty(G^\ell)\ge c_{\mathrm{obs}}.
\end{equation}
Then for every element of $\calK_k$,
\begin{equation}\label{eq:detector-dichotomy-template}
\Psi(r_k)\ge\eps_0
\quad\Longrightarrow\quad
A_k^\infty(G^\ell)\ge c_{\mathrm{obs}}
\quad\text{or}\quad
\Omega^\ell(r_k)>\eta\eps_0.
\end{equation}
\end{proposition}

\begin{proof}
If $\Psi(r_k)\ge\eps_0$ and $\Omega^\ell(r_k)\le\eta\eps_0$, Corollary~\ref{cor:coarse-visibility} gives $\Psi^\ell(r_k)\ge(1/4-\eta)\eps_0$.  The assumed coarse observability estimate \eqref{eq:coarse-observability-hyp} then gives $A_k^\infty(G^\ell)\ge c_{\mathrm{obs}}$.
\end{proof}

\begin{remark}[Where the real difficulty remains]
The hypothesis \eqref{eq:coarse-observability-hyp} is not a consequence of the resolution lemma.  It is a separate compactness and separation problem.  It can fail through pressure--flux cancellation, harmonic pressure tails, coherent low-frequency resolved flow, or other silent profiles.  The resolution lemma only removes the subfilter ambiguity.
\end{remark}
\subsection{Proof of the Resolution Lemma}

We prove Theorem~\ref{thm:resolution} at the origin.  The proof at a general center $z_0$ is identical.

Write
\[
u=U^\ell+(u-U^\ell).
\]
The elementary inequality
\begin{equation}\label{eq:p-power-ineq-3}
|a+b|^3\le 4(|a|^3+|b|^3)
\end{equation}
gives
\begin{equation}\label{eq:velocity-split-proof}
\int_{Q_r}|u|^3\dxdt
\le
4\int_{Q_r}|U^\ell|^3\dxdt
+
4\int_{Q_r}|u-U^\ell|^3\dxdt.
\end{equation}

For the pressure, use the projection $\Pf_r$ from \eqref{eq:pressure-projection}.  Since
\[
p=P^\ell+(p-P^\ell),
\]
linearity gives
\begin{equation}\label{eq:pressure-exact-split}
\Pf_r p=\Pf_r P^\ell+\Pf_r(p-P^\ell).
\end{equation}
Using
\begin{equation}\label{eq:p-power-ineq-32}
|a+b|^{3/2}\le 2^{1/2}\left(|a|^{3/2}+|b|^{3/2}\right),
\end{equation}
we obtain
\begin{align}\label{eq:pressure-split-proof}
\int_{Q_r}|\Pf_r p|^{3/2}\dxdt
&\le
2^{1/2}\int_{Q_r}|\Pf_r P^\ell|^{3/2}\dxdt
\nonumber\\
&\quad+
2^{1/2}\int_{Q_r}|\Pf_r(p-P^\ell)|^{3/2}\dxdt.
\end{align}
Adding \eqref{eq:velocity-split-proof} and \eqref{eq:pressure-split-proof}, and dividing by $r^2$, proves the sharper estimate \eqref{eq:sharp-resolution-estimate}.  Since $2^{1/2}\le4$, the simpler bound \eqref{eq:resolution-estimate} follows.

\begin{proof}[Proof of Corollary~\ref{cor:coarse-visibility}]
By Theorem~\ref{thm:resolution},
\[
\eps_0\le\Psi(r)\le4\Psi^\ell(r)+4\Omega^\ell(r)
\le4\Psi^\ell(r)+4\eta\eps_0.
\]
Hence
\[
4\Psi^\ell(r)\ge(1-4\eta)\eps_0,
\]
which is \eqref{eq:coarse-lower-bound}.
\end{proof}

\begin{proof}[Proof of Corollary~\ref{cor:finite-chain}]
For each $k\in\mathcal B_{\eps_0}$, either $\Omega^\ell(r_k)>\eta\eps_0$, in which case $k\in\mathcal R^\ell_{\eta\eps_0}$, or $\Omega^\ell(r_k)\le\eta\eps_0$, in which case Corollary~\ref{cor:coarse-visibility} gives $k\in\mathcal B^\ell_{c_0\eps_0}$.  This proves \eqref{eq:bad-set-inclusion}.  The counting and weighted estimates follow immediately.
\end{proof}
\subsection{Quantifying the Subfilter Residual}

The resolution lemma is useful only to the extent that $\Omega^\ell$ can be controlled or recognized as an obstruction.  This section gives elementary residual bounds in terms of local translation moduli.

\subsubsection{Spatial increment moduli}

Let $I_r(t_0)=(t_0-r^2,t_0)$ and let
\[
Q_r^\ell(z_0)=B_{r+\ell}(x_0)\times I_r(t_0).
\]
For $f\in L^q(Q_r^\ell(z_0))$, define the local spatial increment modulus
\begin{equation}\label{eq:increment-modulus}
\mathfrak m_q(f;z_0,r,\ell)=
r^{-2}\sup_{|h|\le\ell}\int_{I_r(t_0)}\int_{B_r(x_0)}
|f(x+h,t)-f(x,t)|^q\dxdt.
\end{equation}
Here $f$ is interpreted on the enlarged spatial cylinder, so the shifted point $x+h$ belongs to $B_{r+\ell}(x_0)$.

\begin{lemma}[Mollification error by increments]\label{lem:mollifier-increments}
Let $1\le q<\infty$.  If $S_\ell$ is the spatial mollifier defined above, then
\begin{equation}\label{eq:mollifier-error-increments}
r^{-2}\int_{Q_r(z_0)}|f-S_\ell f|^q\dxdt
\le
\mathfrak m_q(f;z_0,r,\ell).
\end{equation}
Moreover,
\begin{equation}\label{eq:projected-mollifier-error-increments}
r^{-2}\int_{Q_r(z_0)}|\Pf_{r,z_0}(f-S_\ell f)|^q\dxdt
\le
C_q\mathfrak m_q(f;z_0,r,\ell).
\end{equation}
\end{lemma}

\begin{proof}
For $(x,t)\in Q_r(z_0)$,
\[
f(x,t)-S_\ell f(x,t)=\int_{B_\ell}\rho_\ell(y)\bigl(f(x,t)-f(x-y,t)\bigr)\dy.
\]
By Jensen's inequality and $\int\rho_\ell=1$,
\[
|f(x,t)-S_\ell f(x,t)|^q
\le
\int_{B_\ell}\rho_\ell(y)|f(x,t)-f(x-y,t)|^q\dy.
\]
Integrating over $Q_r(z_0)$ and taking the supremum over $|y|\le\ell$ proves \eqref{eq:mollifier-error-increments}.  For the projected estimate, use the elementary bound
\[
\norm{g-(g)_{B_r}(t)}{L^q(B_r)}\le 2\norm{g}{L^q(B_r)}
\]
for a.e. $t$, then integrate in time.
\end{proof}

\begin{proposition}[Residual bound by local increments]\label{prop:residual-increments}
For every $Q_r(z_0)$ on which the following quantities are finite,
\begin{equation}\label{eq:Omega-increment-bound}
\Omega^\ell(z_0,r)
\le
\mathfrak m_3(u;z_0,r,\ell)+C\mathfrak m_{3/2}(p;z_0,r,\ell).
\end{equation}
The constant $C$ is universal.
\end{proposition}

\begin{proof}
Apply Lemma~\ref{lem:mollifier-increments} with $q=3$ to $u$ and with $q=3/2$ to $p$, using the definition \eqref{eq:Omega-residual}.
\end{proof}

\subsubsection{A scale-invariant Besov corollary}

The preceding estimate can be converted into explicit powers of $\ell/r$ if spatial increments are controlled.

\begin{corollary}[Besov-type residual smallness]\label{cor:besov-residual}
Suppose that for some $\alpha,\beta\in(0,1]$ and constants $M_u,M_p$ one has
\begin{align}
\sup_{|h|\le\ell}\int_{Q_r^\ell(z_0)}|u(x+h,t)-u(x,t)|^3\dxdt
&\le M_u^3\left(\frac{\ell}{r}\right)^{3\alpha}r^2,
\label{eq:u-besov-assumption}\\
\sup_{|h|\le\ell}\int_{Q_r^\ell(z_0)}|p(x+h,t)-p(x,t)|^{3/2}\dxdt
&\le M_p^{3/2}\left(\frac{\ell}{r}\right)^{\frac32\beta}r^2.
\label{eq:p-besov-assumption}
\end{align}
Then
\begin{equation}\label{eq:besov-Omega-bound}
\Omega^\ell(z_0,r)
\le
C\left[M_u^3\left(\frac{\ell}{r}\right)^{3\alpha}
+M_p^{3/2}\left(\frac{\ell}{r}\right)^{\frac32\beta}\right].
\end{equation}
In particular, if $\ell/r$ is sufficiently small in terms of $M_u,M_p,\alpha,\beta$ and $\eps_0$, then the residual alternative in Corollary~\ref{cor:coarse-visibility} is excluded.
\end{corollary}

\begin{remark}[No regularity is hidden in the main theorem]
Corollary~\ref{cor:besov-residual} is an optional refinement.  The resolution estimate itself does not assume these increment bounds.  When no such information is available, a large value of $\Omega^\ell$ is a genuine obstruction rather than a proof defect.
\end{remark}
\subsection{Pressure Residuals and Harmonic Tails}\label{sec:pressure-residuals}

The pressure part of $\Omega^\ell$ is often the most delicate term.  This section records a standard way to separate local Calderon--Zygmund pressure from harmonic pressure.  The goal is not to prove a new pressure theorem, but to identify the pieces that the residual channel contains.

\subsubsection{Local pressure split}

Let $B_r(x_0)\Subset B_R(x_0)$ and let $\chi\in C_c^\infty(B_R(x_0))$ satisfy $\chi\equiv1$ on a neighborhood of $B_r(x_0)$.  For a.e. $t$, define the localized pressure part
\begin{equation}\label{eq:localized-pressure-part}
p_{\loc}(\cdot,t)=\mathcal R_i\mathcal R_j\bigl(\chi u_i u_j\bigr)(\cdot,t),
\end{equation}
where $\mathcal R_i$ are the Riesz transforms on $\R^3$, and summation over $i,j$ is understood.  Then
\[
p_h=p-p_{\loc}
\]
is harmonic in $B_r(x_0)$ for a.e. $t$, after the usual pressure normalization.

Thus
\begin{equation}\label{eq:pressure-residual-split}
p-P^\ell
=(p_{\loc}-S_\ell p_{\loc})+(p_h-S_\ell p_h),
\end{equation}
where the identity is understood on cylinders where the convolution is defined.  Applying the pressure projection $\Pf_{r,z_0}$ and the triangle inequality gives
\begin{align}\label{eq:pressure-residual-local-harmonic}
r^{-2}\int_{Q_r(z_0)}|\Pf_r(p-P^\ell)|^{3/2}\dxdt
&\le
C r^{-2}\int_{Q_r(z_0)}|\Pf_r(p_{\loc}-S_\ell p_{\loc})|^{3/2}\dxdt
\nonumber\\
&\quad+
C r^{-2}\int_{Q_r(z_0)}|\Pf_r(p_h-S_\ell p_h)|^{3/2}\dxdt.
\end{align}
The first term is a localized singular-integral residual; the second is a harmonic tail residual.

\subsubsection{Harmonic smoothing bound}

The harmonic term can be estimated by interior regularity.  The following simple form is sufficient for interpretation.

\begin{lemma}[Harmonic residual estimate]\label{lem:harmonic-residual}
Let $h(\cdot,t)$ be harmonic in $B_R(x_0)$ for a.e. $t\in I_r(t_0)$, and assume $0<\ell\le (R-r)/4$.  Then
\begin{equation}\label{eq:harmonic-residual-bound}
r^{-2}\int_{Q_r(z_0)}|\Pf_r(h-S_\ell h)|^{3/2}\dxdt
\le
C\left(\frac{\ell}{R-r}\right)^{3/2}
 r^{-2}\int_{I_r(t_0)}\int_{B_R(x_0)}|h-(h)_{B_R}(t)|^{3/2}\dxdt.
\end{equation}
The constant is universal up to the fixed mollifier.
\end{lemma}

\begin{proof}
For $(x,t)\in Q_r(z_0)$,
\[
h(x,t)-S_\ell h(x,t)=\int\rho_\ell(y)(h(x,t)-h(x-y,t))\dy.
\]
Since $h$ is harmonic in $B_R(x_0)$ and $x,x-y\in B_{(R+r)/2}(x_0)$, the interior gradient estimate gives
\[
|h(x,t)-h(x-y,t)|
\le
C\frac{|y|}{R-r}\left(\frac{1}{|B_R|}\int_{B_R(x_0)}|h(\xi,t)-(h)_{B_R}(t)|^{3/2}\,d\xi\right)^{2/3}.
\]
Jensen's inequality and $|y|\le\ell$ on the support of $\rho_\ell$ yield
\[
|h-S_\ell h|^{3/2}
\le
C\left(\frac{\ell}{R-r}\right)^{3/2}
\frac{1}{|B_R|}\int_{B_R(x_0)}|h-(h)_{B_R}(t)|^{3/2}\,d\xi.
\]
Integrating over $B_r\times I_r$ and applying the projection bound as in Lemma~\ref{lem:mollifier-increments} proves the estimate after adjusting constants.
\end{proof}

\begin{remark}[Harmonic pressure as a silent channel]
The harmonic residual is not a removable gauge artifact.  A harmonic pressure component can contribute to local pressure work and may remain significant across several scales.  In a detector-to-CKN argument, this term should be recorded explicitly, not hidden inside an unspecified error.
\end{remark}

\subsubsection{Local part and velocity increments}

For the localized pressure $p_{\loc}$ in \eqref{eq:localized-pressure-part}, Calderon--Zygmund boundedness gives
\begin{equation}\label{eq:CZ-pressure-bound}
\norm{p_{\loc}(\cdot,t)}{L^{3/2}(\R^3)}
\le C\norm{u(\cdot,t)}{L^3(B_R)}^2.
\end{equation}
A quantitative increment estimate for $p_{\loc}-S_\ell p_{\loc}$ can be obtained either directly from increments of $p_{\loc}$ or indirectly from commutator estimates for $u\otimes u$.  A basic version is
\begin{equation}\label{eq:ploc-increment-bound}
r^{-2}\int_{Q_r}|\Pf_r(p_{\loc}-S_\ell p_{\loc})|^{3/2}\dxdt
\le C\mathfrak m_{3/2}(p_{\loc};z_0,r,\ell).
\end{equation}
Under additional Besov-type control on $u$, this can be bounded in terms of velocity increments.  Such estimates are standard in coarse-graining arguments, but their exact form depends on the chosen function spaces and cutoffs.  The important point for the present paper is structural: the pressure residual splits into a local singular-integral residual and a harmonic residual.
\subsection{Examples and Necessity of the Residual Alternative}

The following examples are not asserted to be Navier--Stokes solutions.  They show that the form of the resolution lemma is forced at the level of the scale-critical functionals.

\begin{example}[Unresolved oscillation]
Let $u_\delta(x)=a\sin(x_1/\delta)e_1$ on a fixed ball, with $\delta\ll\ell\ll r$, and take $p=0$.  The full velocity contribution to $\Psi(r)$ is of size $|a|^3$, while $S_\ell u_\delta$ is small because the oscillation averages out.  Thus $\Psi^\ell(r)$ can be small although $\Psi(r)$ is not.  The missing contribution is exactly $r^{-2}\int|u-S_\ell u|^3$, hence belongs to $\Omega^\ell(r)$.
\end{example}

\begin{example}[Pressure oscillation below the coarse scale]
Let $u=0$ and $p_\delta(x)=b\sin(x_1/\delta)$, with $\delta\ll\ell\ll r$.  Then the full pressure term $r^{-2}\int|\Pf_r p_\delta|^{3/2}$ is nontrivial, but the coarse pressure $P^\ell=S_\ell p_\delta$ is small.  Again the loss is not a failure of the proof; it is the pressure component of $\Omega^\ell$.
\end{example}

\begin{example}[Large smooth resolved fields]
A smooth low-frequency field may have large $\Psi^\ell(r)$ on a moderately large cylinder even though it has no singular behavior.  The resolution lemma deliberately does not distinguish singular concentration from large but smooth resolved size.  That distinction belongs to a later observability or decay argument, not to the resolution step.
\end{example}
\section{Pressure--Flux Work Depletion}\label{sec:work-depletion}

\subsection{Setup and main theorem}\label{sec:setup}

\subsubsection{Local solution and common spatial coarse graining}

Let
\[
\mathcal Q=B_{2r_0}(x_0)\times(T_0,T_1)
\]
with \(r_0>0\), and let \((u,p)\) be a suitable weak solution of \eqref{eq:NS} in \(\mathcal Q\). Thus, locally,
\begin{equation}\label{eq:energy-class}
u\in L^\infty_tL^2_x\cap L^2_tH^1_x,
\qquad
p\in L^{3/2}.
\end{equation}
We use the standard weakly continuous representative of \(u\). On bounded subcylinders, interpolation gives \(u\in L^3\).

Choose a nonnegative radial function \(\rho\in C_c^\infty(B_1)\) with \(\int\rho=1\), and write
\[
\rho_\ell(x)=\ell^{-3}\rho(x/\ell),
\qquad
S_\ell f=\rho_\ell*_x f.
\]
Fix one physical length \(\ell>0\), small enough that every spatial convolution used below remains inside \(B_{2r_0}(x_0)\). Define
\begin{equation}\label{eq:coarse-package}
U:=S_\ell u,
\qquad
P:=S_\ell p,
\qquad
R:=S_\ell(u\otimes u)-U\otimes U,
\qquad
\Pi:=-R:\nabla U.
\end{equation}
The superscript \(\ell\) will occasionally be restored when the dependence on the filter length matters.

\subsubsection{A finite chain of adjacent slabs}

Fix \(N\in\mathbb N\), \(0<\theta<1\), and set
\begin{equation}\label{eq:radii}
r_k=\theta^k r_0,
\qquad
k=0,\ldots,N.
\end{equation}
Choose adjacent times
\begin{equation}\label{eq:times}
\tau_{k+1}-\tau_k=r_k^2,
\qquad
k=0,\ldots,N-1,
\end{equation}
so that every closed slab
\[
\overline{Q_k}
:=
\overline{B_{r_k}(x_0)}\times[\tau_k,\tau_{k+1}]
\]
is compactly contained in the region on which the coarse package is defined.

Let \(\chi\in C_c^\infty(B_1)\) be radial, nonincreasing in \(|x|\), and satisfy
\begin{equation}\label{eq:base-cutoff}
0\le \chi\le1,
\qquad
\chi\equiv1\text{ on }B_\vartheta
\end{equation}
for some \(0<\vartheta<1\). Set
\begin{equation}\label{eq:chi-k}
\chi_k(x)=\chi\left(\frac{x-x_0}{r_k}\right).
\end{equation}
Then
\begin{equation}\label{eq:nesting}
\chi_{k+1}\le\chi_k
\qquad\text{pointwise.}
\end{equation}

\subsubsection{Local work, dissipation, and leakage}

Let \(I=(t_-,t_+)\), let \(r>0\), and let
\[
\phi\in C^\infty([t_-,t_+];C_c^\infty(B_r(x_0)))
\]
be nonnegative. Define the localized kinetic energy
\begin{equation}\label{eq:Kphi}
K_\phi(t)=\int_{\R^3}\frac12|U(x,t)|^2\phi(x,t)\dx.
\end{equation}
Define the normalized combined work, resolved dissipation, and localization functional by
\begin{align}
\W_{I,r}[\phi]
&:=
r^{-1}\int_I\int
\left(
\phi\Pi-PU\cdot\nabla\phi
\right)\dxdt,
\label{eq:W-def}
\\
\D_{I,r}[\phi]
&:=
r^{-1}\int_I\int
\phi|\nabla U|^2\dxdt,
\label{eq:D-def}
\\
\Lk_{I,r}[\phi]
&:=
r^{-1}\int_I\int
\left[
\frac12|U|^2(\partial_t\phi+\Delta\phi)
+
\frac12|U|^2U\cdot\nabla\phi
+
(RU)\cdot\nabla\phi
\right]\dxdt.
\label{eq:L-def}
\end{align}
The distribution
\begin{equation}\label{eq:G-def}
G:=\Pi+\nabla\cdot(PU)
\end{equation}
is understood through
\begin{equation}\label{eq:G-pairing}
\pair{G}{\phi}
=
\int\phi\Pi\dxdt-
\int PU\cdot\nabla\phi\dxdt.
\end{equation}
Thus \(\W_{I,r}[\phi]=r^{-1}\pair{G}{\phi}\).

For the slab \(Q_k\), we abbreviate
\[
\W_k[\phi]=\W_{I_k,r_k}[\phi],
\qquad
\D_k[\phi]=\D_{I_k,r_k}[\phi],
\qquad
\Lk_k[\phi]=\Lk_{I_k,r_k}[\phi].
\]
The endpoint quantities use only the base spatial cutoff:
\begin{equation}\label{eq:E-endpoints}
\Ek_k^-:=r_k^{-1}K_{\chi_k}(\tau_k),
\qquad
\Ek_k^+:=r_k^{-1}K_{\chi_k}(\tau_{k+1}).
\end{equation}

\subsubsection{Finite-dimensional active work coefficients}

For each \(k\), fix an integer \(m_k\ge1\) and profiles
\begin{equation}\label{eq:profiles}
\psi_{k,1},\ldots,\psi_{k,m_k}
\in C_c^\infty(B_1\times(0,1)).
\end{equation}
Their physical pullbacks are
\begin{equation}\label{eq:pullback}
\psi_{k,j}^\sharp(x,t)
=
\psi_{k,j}\left(
\frac{x-x_0}{r_k},
\frac{t-\tau_k}{r_k^2}
\right).
\end{equation}
Define the coefficient vector \(g_k=(g_{k,1},\ldots,g_{k,m_k})\) by
\begin{equation}\label{eq:gcoeff}
g_{k,j}
:=
r_k^{-1}\pair{G}{\chi_k\psi_{k,j}^\sharp}
=
r_k^{-1}\int_{Q_k}
\left[
\chi_k\psi_{k,j}^\sharp\Pi
-
PU\cdot\nabla(\chi_k\psi_{k,j}^\sharp)
\right]\dxdt.
\end{equation}
The active work coefficient norm is
\begin{equation}\label{eq:active-norm}
\Aact_k(G):=|g_k|_{\ell^2}.
\end{equation}
This is a finite family of continuous distributional pairings. It is deliberately defined without assuming \(G\in L^2\). If \(G\) has additional square integrability and the profiles are chosen as a weighted orthonormal family, \eqref{eq:active-norm} can be identified with a finite-dimensional projection norm.

Choose an invertible matrix
\begin{equation}\label{eq:Bmatrix}
B_k=(b^k_{\alpha j})_{\alpha,j=1}^{m_k}
\in\R^{m_k\times m_k}
\end{equation}
and set
\begin{equation}\label{eq:eta-alpha}
\eta_{k,\alpha}
=
\sum_{j=1}^{m_k}b^k_{\alpha j}\psi_{k,j}.
\end{equation}
Choose \(\eps_k>0\) so that
\begin{equation}\label{eq:eps-positive}
\eps_k\max_{1\le\alpha\le m_k}
\norm{\eta_{k,\alpha}}{L^\infty(B_1\times(0,1))}
\le\frac12.
\end{equation}
Define \(m_k+1\) nonnegative weights on \(Q_k\) by
\begin{align}
\phi_{k,0}(x,t)&=\chi_k(x),
\label{eq:phi0}
\\
\phi_{k,\alpha}(x,t)
&=
\chi_k(x)
\left[
1+\eps_k\eta_{k,\alpha}\left(
\frac{x-x_0}{r_k},
\frac{t-\tau_k}{r_k^2}
\right)
\right],
\quad 1\le\alpha\le m_k.
\label{eq:phialpha}
\end{align}
Because the profiles are compactly supported in the open reference time interval,
\begin{equation}\label{eq:common-endpoints}
\phi_{k,\alpha}(\cdot,\tau_k)
=
\phi_{k,\alpha}(\cdot,\tau_{k+1})
=
\chi_k
\end{equation}
for every \(\alpha=0,\ldots,m_k\).

Set
\begin{equation}\label{eq:ck}
c_k
:=
\frac{\eps_k}{2\sqrt{m_k}}\,
\sigma_{\min}(B_k),
\end{equation}
where \(\sigma_{\min}(B_k)>0\) is the smallest singular value.

\subsubsection{Pressure--flux cancellation ledger}

For any admissible \(\phi\), define the two signed channels
\begin{equation}\label{eq:channels}
\mathcal F_{I,r}[\phi]
:=
r^{-1}\int_I\int\phi\Pi\dxdt,
\qquad
\mathcal P_{I,r}[\phi]
:=
-r^{-1}\int_I\int PU\cdot\nabla\phi\dxdt.
\end{equation}
Then \(\W_{I,r}[\phi]=\mathcal F_{I,r}[\phi]+\mathcal P_{I,r}[\phi]\). The nonnegative cancellation quantity is
\begin{equation}\label{eq:cancellation}
\Cpf_{I,r}[\phi]
:=
\abs{\mathcal F_{I,r}[\phi]}
+
\abs{\mathcal P_{I,r}[\phi]}
-
\abs{\W_{I,r}[\phi]}
\ge0.
\end{equation}
Thus individual pressure and flux activity that disappears from the combined observable is recorded by \(\Cpf\), not absorbed into localization error.

\subsubsection{Main theorem}

\begin{theorem}[Combined pressure--flux work depletion on a fixed finite chain]\label{thm:main}
Assume the setup above. For every \(k=0,\ldots,N-1\), there is an index
\[
\beta_k\in\{0,1,\ldots,m_k\}
\]
such that, with
\begin{equation}\label{eq:selected}
\widehat\phi_k:=\phi_{k,\beta_k},
\qquad
\W_k:=\W_k[\widehat\phi_k],
\qquad
\D_k:=\D_k[\widehat\phi_k],
\qquad
\Lk_k:=\Lk_k[\widehat\phi_k],
\end{equation}
one has
\begin{equation}\label{eq:extraction-main}
|\W_k|\ge c_k\Aact_k(G).
\end{equation}
Define
\begin{equation}\label{eq:Wpm}
\W_k^+=\max\{\W_k,0\},
\qquad
\W_k^-=\max\{-\W_k,0\},
\end{equation}
and weights
\begin{equation}\label{eq:wk}
w_k=\frac{r_k}{r_0}.
\end{equation}
Then
\begin{equation}\label{eq:main-chain}
\boxed{
\sum_{k=0}^{N-1}w_k\bigl(\W_k^++\D_k\bigr)
\le
\Ek_0^-
+
\sum_{k=0}^{N-1}w_k|\Lk_k|
+
\sum_{k=0}^{N-1}w_k\W_k^-.
}
\end{equation}
In particular, if
\begin{equation}\label{eq:F-B-sets}
\mathcal I_+
=\{k:\W_k\ge0\},
\qquad
\mathcal I_-
=\{k:\W_k<0\},
\end{equation}
then
\begin{equation}\label{eq:main-active}
\boxed{
\sum_{k\in\mathcal I_+}w_kc_k\Aact_k(G)
+
\sum_{k=0}^{N-1}w_k\D_k
\le
\Ek_0^-
+
\sum_{k=0}^{N-1}w_k|\Lk_k|
+
\sum_{k\in\mathcal I_-}w_k\W_k^-.
}
\end{equation}
For every \(k\in\mathcal I_-\), the explicit backscatter quantity satisfies
\begin{equation}\label{eq:backscatter-main}
\Bback_k^{\mathrm{back}}
:=\W_k^-
\ge c_k\Aact_k(G).
\end{equation}
All quantities in \eqref{eq:main-chain}--\eqref{eq:backscatter-main} are dimensionless under Navier--Stokes scaling.
\end{theorem}

\begin{remark}[What the theorem does and does not say]\label{rem:scope}
The theorem is unconditional at fixed \(N\), fixed \(\ell>0\), and fixed active profiles. It does not assert that \(\mathcal I_-\) is small, that \(c_k\) is uniform in a moving-window limit, or that \(\sum_k w_k|\Lk_k|\) is summable as \(N\to\infty\). Such statements require additional PDE information not contained in the local energy identity.
\end{remark}

\subsection{The Navier--Stokes-generated coarse package}\label{sec:package}

\begin{lemma}[Common coarse package]\label{lem:coarse-package}
On every interior subcylinder whose spatial \(\ell\)-neighborhood is contained in \(B_{2r_0}(x_0)\), the fields in \eqref{eq:coarse-package} satisfy
\begin{equation}\label{eq:coarse-momentum}
\partial_tU-\Delta U+\nabla\cdot(U\otimes U)+\nabla P
=-\nabla\cdot R,
\qquad
\nabla\cdot U=0
\end{equation}
in distributions. Moreover, \(R=R^T\ge0\) pointwise almost everywhere, and locally
\begin{align}
\norm{U}{L^3}&\le\norm{u}{L^3},
\label{eq:U-L3}
\\
\norm{P}{L^{3/2}}&\le\norm{p}{L^{3/2}},
\label{eq:P-L32}
\\
\norm{R}{L^{3/2}}&\le2\norm{u}{L^3}^2,
\label{eq:R-L32}
\\
\norm{\nabla U}{L^3}&\le C\ell^{-1}\norm{u}{L^3}.
\label{eq:gradU-L3}
\end{align}
Consequently,
\begin{equation}\label{eq:Pi-L1}
\Pi\in L^1_{\mathrm{loc}},
\qquad
PU\in L^1_{\mathrm{loc}},
\qquad
RU\in L^1_{\mathrm{loc}}.
\end{equation}
Finally, \(U\) has a representative in \(C_tL^2_x\) on compact interior spatial sets. In \eqref{eq:U-L3}--\eqref{eq:gradU-L3}, the left-hand norms may be taken on any fixed interior cylinder and the right-hand norms on its spatial \(\ell\)-neighborhood.
\end{lemma}

\begin{proof}
Spatial convolution of \eqref{eq:NS} gives
\[
\partial_tU-\Delta U+\nabla\cdot S_\ell(u\otimes u)+\nabla P=0,
\qquad
\nabla\cdot U=0.
\]
Since
\[
S_\ell(u\otimes u)=U\otimes U+R,
\]
we obtain \eqref{eq:coarse-momentum}.

Symmetry of \(R\) is immediate. For \(\xi\in\R^3\),
\begin{align*}
\xi\cdot R(x,t)\xi
&=
S_\ell\bigl((\xi\cdot u)^2\bigr)(x,t)
-
\bigl(S_\ell(\xi\cdot u)(x,t)\bigr)^2
\\
&\ge0
\end{align*}
by Jensen's inequality, because \(\rho_\ell\) is nonnegative and has unit mass. Thus \(R\) is a covariance tensor and is positive semidefinite.

The estimates \eqref{eq:U-L3} and \eqref{eq:P-L32} follow from Young's inequality. Pointwise,
\[
|R|
\le
S_\ell(|u|^2)+|U|^2,
\]
so another application of Young's inequality gives \eqref{eq:R-L32}. Since
\[
\nabla U=(\nabla\rho_\ell)*u,
\qquad
\norm{\nabla\rho_\ell}{L^1}
=\ell^{-1}\norm{\nabla\rho}{L^1},
\]
we obtain \eqref{eq:gradU-L3}. H\"older's inequality then yields \eqref{eq:Pi-L1}.

A local energy-class weak solution has a weakly continuous \(L^2\)-representative. For a compact set \(K\) whose \(\ell\)-neighborhood is interior, the map
\[
T_\ell:L^2\longrightarrow L^2(K),
\qquad
T_\ell f=(\rho_\ell*f)|_K,
\]
is compact, because it maps bounded sets into bounded subsets of every spatial Sobolev space on \(K\). A compact linear map sends weakly convergent sequences to strongly convergent sequences. Applying this observation to \(u(t_n)\rightharpoonup u(t)\) shows that \(U(t_n)\to U(t)\) strongly in \(L^2(K)\). Hence \(U\in C_tL^2_x(K)\).
\end{proof}

\begin{remark}
The use of one common \(\ell\) is structural. If a different coarse field is introduced at every scale, endpoint energies generally belong to different resolved systems and there is no automatic common telescoping budget.
\end{remark}

\subsection{Local combined work identity and pressure gauge}\label{sec:identity}

\begin{lemma}[Local coarse energy identity with endpoint traces]\label{lem:local-identity}
Let \(I=(t_-,t_+)\) and let \(\phi\in C^\infty([t_-,t_+];C_c^\infty)\) be nonnegative. Then
\begin{equation}\label{eq:local-identity}
\boxed{
\W_{I,r}[\phi]
+
\D_{I,r}[\phi]
=
r^{-1}\bigl(K_\phi(t_-)-K_\phi(t_+)\bigr)
+
\Lk_{I,r}[\phi].
}
\end{equation}
Equivalently, the coarse fields satisfy the distributional local balance
\begin{equation}\label{eq:pointwise-energy}
\partial_t\frac{|U|^2}{2}
-
\Delta\frac{|U|^2}{2}
+
|\nabla U|^2
+
\nabla\cdot
\left[
\left(\frac{|U|^2}{2}+P\right)U+RU
\right]
=-\Pi.
\end{equation}
\end{lemma}

\begin{proof}
The formal computation follows by taking the scalar product of \eqref{eq:coarse-momentum} with \(U\). The only point is justification in time. We give the standard Steklov-average argument.

For \(h>0\), write
\[
[f]_h(t)=\frac1h\int_t^{t+h}f(s)\dd s
\]
on a slightly shorter interval. Apply this averaging to \eqref{eq:coarse-momentum} and test the averaged equation with \([U]_h\phi\). Every term is integrable by \Cref{lem:coarse-package}. Integrating by parts in space and time gives the classical averaged identity. As \(h\downarrow0\),
\[
[U]_h\to U\quad\text{in }L^3_{\mathrm{loc}},
\qquad
\nabla[U]_h\to\nabla U\quad\text{in }L^2_{\mathrm{loc}},
\]
while
\[
[U\otimes U]_h\to U\otimes U,
\quad
[R]_h\to R,
\quad
[P]_h\to P
\]
in their natural local Lebesgue spaces. The endpoint terms converge because \(U\in C_tL^2_x\) after spatial coarse graining. Passing to the limit yields \eqref{eq:pointwise-energy}.

Multiplying \eqref{eq:pointwise-energy} by \(\phi\), integrating over \(I\times\R^3\), and integrating by parts gives
\begin{align*}
&K_\phi(t_+)-K_\phi(t_-)
+
\int_I\int\phi|\nabla U|^2\dxdt
+
\int_I\int\phi\Pi\dxdt
\\
&\quad=
\int_I\int
\frac12|U|^2(\partial_t\phi+\Delta\phi)\dxdt
+
\int_I\int
\left[
\left(\frac12|U|^2+P\right)U+RU
\right]\cdot\nabla\phi\dxdt.
\end{align*}
Move the pressure-transport term to the left and divide by \(r\). This is exactly \eqref{eq:local-identity}.
\end{proof}

\begin{remark}[Why the temporal term is necessary]
If \(\phi\) depends on time, the term
\[
r^{-1}\int\frac12|U|^2\partial_t\phi
\]
is part of the exact PDE identity. Space-time work extraction without this term is incompatible with the localized energy equation. The endpoint condition \eqref{eq:common-endpoints} ensures that this extra localization cost does not alter the endpoint profiles used in the finite-chain sum.
\end{remark}

\begin{lemma}[Pressure-gauge invariance]\label{lem:gauge}
Let \(a\in L^{3/2}_{\mathrm{loc}}(I)\) depend only on time. Replacing \(P\) by \(P+a(t)\) leaves \(\W_{I,r}[\phi]\) unchanged.
\end{lemma}

\begin{proof}
The change in \(\W_{I,r}[\phi]\) is
\[
-r^{-1}\int_I a(t)
\left(\int U(x,t)\cdot\nabla\phi(x,t)\dx\right)\dt.
\]
For almost every \(t\),
\[
\int U\cdot\nabla\phi\dx
=-\int\phi\nabla\cdot U\dx=0.
\]
Thus the change vanishes. Notice that a spatially harmonic pressure is not, in general, a gauge: only functions of time disappear in this way.
\end{proof}

\subsection{Explicit localization leakage}\label{sec:leakage}

For a nonnegative weight on \(Q_r=B_r(x_0)\times I\), define the scale-adapted seminorm
\begin{equation}\label{eq:phi-seminorm}
[\phi]_{\mathfrak C(r)}
:=
\norm{\phi}{L^\infty}
+
r\norm{\nabla\phi}{L^\infty}
+
r^2\left(
\norm{\partial_t\phi}{L^\infty}
+
\norm{\nabla^2\phi}{L^\infty}
\right).
\end{equation}
Let
\begin{equation}\label{eq:derivative-region}
\mathscr A_\phi
:=
\supp(\partial_t\phi)
\cup
\supp(\nabla\phi)
\cup
\supp(\nabla^2\phi).
\end{equation}

\begin{lemma}[Scale-invariant leakage bound]\label{lem:leakage}
If \([\phi]_{\mathfrak C(r)}\le M_\phi\), then
\begin{equation}\label{eq:leakage-bound}
\boxed{
|\Lk_{I,r}[\phi]|
\le
C M_\phi
\left[
 r^{-3}\int_{\mathscr A_\phi}|U|^2\dxdt
+
 r^{-2}\int_{\mathscr A_\phi}|U|^3\dxdt
+
 r^{-2}\int_{\mathscr A_\phi}|R|\,|U|\dxdt
\right].
}
\end{equation}
For the weights \(\phi_{k,\alpha}\) in \eqref{eq:phi0}--\eqref{eq:phialpha}, the constants \(M_{\phi_{k,\alpha}}\) are bounded by an explicit number depending only on \(\chi\), \(\eps_k\), \(B_k\), and finitely many \(C^2\)-norms of the reference profiles.
\end{lemma}

\begin{proof}
From \eqref{eq:L-def} and \eqref{eq:phi-seminorm},
\[
|\partial_t\phi+\Delta\phi|
\le C M_\phi r^{-2},
\qquad
|\nabla\phi|
\le M_\phi r^{-1}.
\]
Substitution into \eqref{eq:L-def} gives \eqref{eq:leakage-bound}. For \(\phi_{k,\alpha}\), the chain rule under the parabolic map \eqref{eq:pullback} yields
\[
|\partial_t\phi_{k,\alpha}|+|\nabla^2\phi_{k,\alpha}|
\le C_k r_k^{-2},
\qquad
|\nabla\phi_{k,\alpha}|
\le C_k r_k^{-1},
\]
with the stated explicit dependence.
\end{proof}

\begin{remark}[Purely spatial cutoffs]
For \(\phi(x,t)=\chi(x)\), the temporal contribution vanishes and \(\mathscr A_\phi\) is the spatial transition annulus. For active space-time weights, derivatives of the interior profiles create an additional, explicitly displayed interior localization region.
\end{remark}

\begin{remark}[Scaling]
Under
\[
U_\lambda(x,t)=\lambda U(\lambda x,\lambda^2t),
\quad
P_\lambda(x,t)=\lambda^2P(\lambda x,\lambda^2t),
\quad
R_\lambda(x,t)=\lambda^2R(\lambda x,\lambda^2t),
\]
the three terms on the right of \eqref{eq:leakage-bound} are invariant. Thus the normalization in \eqref{eq:W-def}--\eqref{eq:L-def} is dimensionally consistent.
\end{remark}

\subsection{Constructive active-work extraction}\label{sec:extraction}

\begin{lemma}[Nonnegative finite-dimensional work extraction]\label{lem:extraction}
For every slab \(Q_k\), there exists \(\beta_k\in\{0,1,\ldots,m_k\}\) such that
\begin{equation}\label{eq:extract-lemma}
\abs{\W_k[\phi_{k,\beta_k}]}
\ge
c_k\Aact_k(G),
\end{equation}
where \(c_k\) is given by \eqref{eq:ck}. Consequently, exactly one of the following alternatives holds:
\begin{align}
\W_k[\phi_{k,\beta_k}]
&\ge c_k\Aact_k(G),
\tag{forward}\label{eq:forward-branch}
\\
-\W_k[\phi_{k,\beta_k}]
&\ge c_k\Aact_k(G).
\tag{backscatter}\label{eq:backscatter-branch}
\end{align}
\end{lemma}

\begin{proof}
Linearity of \(\W_k[\cdot]\) and definitions \eqref{eq:gcoeff}, \eqref{eq:eta-alpha}, and \eqref{eq:phialpha} give
\begin{equation}\label{eq:difference-measurement}
\W_k[\phi_{k,\alpha}]-\W_k[\phi_{k,0}]
=
\eps_k\sum_{j=1}^{m_k}b^k_{\alpha j}g_{k,j}.
\end{equation}
In vector form,
\[
\left(
\W_k[\phi_{k,\alpha}]-\W_k[\phi_{k,0}]
\right)_{\alpha=1}^{m_k}
=
\eps_k B_kg_k.
\]
Therefore
\[
\max_{1\le\alpha\le m_k}
\abs{\W_k[\phi_{k,\alpha}]-\W_k[\phi_{k,0}]}
\ge
\frac{\eps_k}{\sqrt{m_k}}
\norm{B_kg_k}{\ell^2}
\ge
\frac{\eps_k\sigma_{\min}(B_k)}{\sqrt{m_k}}
|g_k|_{\ell^2}.
\]
For each \(\alpha\), the triangle inequality implies
\[
\max\left\{
\abs{\W_k[\phi_{k,0}]},
\abs{\W_k[\phi_{k,\alpha}]}
\right\}
\ge
\frac12
\abs{\W_k[\phi_{k,\alpha}]-\W_k[\phi_{k,0}]}.
\]
Choosing an index that realizes the preceding lower bound proves \eqref{eq:extract-lemma}. The sign alternatives are immediate.
\end{proof}

\begin{example}[An explicit matrix constant]\label{ex:matrix}
For
\[
B=
\begin{pmatrix}
2&1&0\\
1&2&1\\
0&1&2
\end{pmatrix},
\]
the singular values are \(2-\sqrt2\), \(2\), and \(2+\sqrt2\). Hence, when \(m=3\),
\[
c
=
\frac{\eps}{2\sqrt3}(2-\sqrt2).
\]
The identity matrix is also admissible and gives \(c=\eps/(2\sqrt m)\).
\end{example}

\begin{lemma}[Perturbative stability of the active window]\label{lem:stability}
Let \(B\in\R^{m\times m}\) be invertible. Suppose a perturbed family of measurements satisfies
\begin{equation}\label{eq:perturbed-difference}
\widetilde\W_\alpha-\widetilde\W_0
=
\eps\bigl((Bg)_\alpha+e_\alpha\bigr),
\qquad
\norm{e}{\ell^2}\le\delta\norm{g}{\ell^2}.
\end{equation}
If \(0\le\delta<\sigma_{\min}(B)\), then one of the \(m+1\) measurements obeys
\begin{equation}\label{eq:stable-extraction}
\max_{0\le\alpha\le m}|\widetilde\W_\alpha|
\ge
\frac{\eps}{2\sqrt m}
\bigl(\sigma_{\min}(B)-\delta\bigr)
\norm{g}{\ell^2}.
\end{equation}
In particular, if the perturbed matrix is \(\widetilde B=B+E\), then
\begin{equation}\label{eq:weyl}
\sigma_{\min}(\widetilde B)
\ge
\sigma_{\min}(B)-\norm{E}{\mathrm{op}}.
\end{equation}
An exact affine parabolic pullback has \(E=0\).
\end{lemma}

\begin{proof}
From \eqref{eq:perturbed-difference},
\[
\norm{Bg+e}{\ell^2}
\ge
\norm{Bg}{\ell^2}-\norm{e}{\ell^2}
\ge
\bigl(\sigma_{\min}(B)-\delta\bigr)
\norm{g}{\ell^2}.
\]
The \(\ell^2\)-to-\(\ell^\infty\) bound and the same two-measurement triangle inequality used in \Cref{lem:extraction} give \eqref{eq:stable-extraction}. Formula \eqref{eq:weyl} is the standard singular-value perturbation inequality.
\end{proof}

\begin{remark}
\Cref{lem:stability} is the correct place to account for chart errors, profile truncation, or an approximate dual family. Such effects change an explicit finite matrix or create an explicit residual vector; they should not be hidden inside an unspecified observability constant.
\end{remark}

\subsection{Weighted telescoping on the finite chain}\label{sec:chain}

\begin{lemma}[Correct finite-chain telescoping]\label{lem:telescoping}
Let \(\widehat\phi_k\) be any nonnegative weights satisfying
\[
\widehat\phi_k(\cdot,\tau_k)
=
\widehat\phi_k(\cdot,\tau_{k+1})
=
\chi_k.
\]
Set \(w_k=r_k/r_0\). Then
\begin{equation}\label{eq:endpoint-telescope}
\sum_{k=0}^{N-1}
w_k(\Ek_k^- -\Ek_k^+)
\le
\Ek_0^-.
\end{equation}
Moreover, if \(\W_k,\D_k,\Lk_k\) are defined using \(\widehat\phi_k\), then
\begin{equation}\label{eq:universal-chain}
\sum_{k=0}^{N-1}w_k(\W_k^++\D_k)
\le
\Ek_0^-
+
\sum_{k=0}^{N-1}w_k|\Lk_k|
+
\sum_{k=0}^{N-1}w_k\W_k^-.
\end{equation}
\end{lemma}

\begin{proof}
Since \(w_kr_k^{-1}=r_0^{-1}\),
\begin{align*}
\sum_{k=0}^{N-1}w_k(\Ek_k^- -\Ek_k^+)
&=
\frac1{r_0}
\sum_{k=0}^{N-1}
\left[
K_{\chi_k}(\tau_k)-K_{\chi_k}(\tau_{k+1})
\right]
\\
&=
\frac1{r_0}
\Bigg[
K_{\chi_0}(\tau_0)
-
K_{\chi_{N-1}}(\tau_N)
\\
&\hspace{3.2cm}
+
\sum_{k=1}^{N-1}
\bigl(
K_{\chi_k}(\tau_k)-K_{\chi_{k-1}}(\tau_k)
\bigr)
\Bigg].
\end{align*}
By \eqref{eq:nesting}, every term in the final sum is nonpositive. The terminal kinetic energy is nonnegative. This proves \eqref{eq:endpoint-telescope}.

By \Cref{lem:local-identity},
\[
\W_k+\D_k
=
\Ek_k^- -\Ek_k^+ +\Lk_k.
\]
Multiply by \(w_k\), sum over every slab, and use \eqref{eq:endpoint-telescope}:
\[
\sum_k w_k(\W_k+\D_k)
\le
\Ek_0^-+\sum_k w_k\Lk_k.
\]
Since \(\W_k=\W_k^+-\W_k^-\), rearrangement and \(\Lk_k\le|\Lk_k|\) give \eqref{eq:universal-chain}.
\end{proof}

\begin{proof}[Proof of \Cref{thm:main}]
Apply \Cref{lem:extraction} on every slab to choose \(\widehat\phi_k\). These weights have the common endpoint traces \eqref{eq:common-endpoints}; hence \Cref{lem:telescoping} gives \eqref{eq:main-chain}. If \(k\in\mathcal I_+\), then \eqref{eq:extraction-main} implies
\[
\W_k^+=\W_k\ge c_k\Aact_k(G).
\]
If \(k\in\mathcal I_-\), then
\[
\W_k^-=-\W_k\ge c_k\Aact_k(G).
\]
Substitution into \eqref{eq:main-chain} proves \eqref{eq:main-active} and \eqref{eq:backscatter-main}.
\end{proof}

\begin{remark}[Why a forward-only subchain does not telescope]\label{rem:subset}
For an arbitrary set \(J\subset\{0,\ldots,N-1\}\), the sum
\[
\sum_{k\in J}w_k(\Ek_k^- -\Ek_k^+)
\]
need not be bounded by \(\Ek_0^-\). The negative intermediate terms that make \eqref{eq:endpoint-telescope} work may be absent. This is why the rigorous chain estimate sums all slabs and moves negative work to the explicit backscatter side.
\end{remark}

\begin{corollary}[No-backscatter finite-chain depletion]\label{cor:no-backscatter}
If every selected slab is in the forward branch, then
\begin{equation}\label{eq:no-backscatter}
\sum_{k=0}^{N-1}w_k
\left(
 c_k\Aact_k(G)+\D_k
\right)
\le
\Ek_0^-
+
\sum_{k=0}^{N-1}w_k|\Lk_k|.
\end{equation}
\end{corollary}

\subsection{Local pressure decomposition and harmonic tails}\label{sec:pressure}

The first theorem is stated for the full pressure work, which is gauge invariant by \Cref{lem:gauge}. We now record how an active/harmonic split can be introduced without discarding physical harmonic pressure.

\begin{lemma}[Active/harmonic pressure split and harmonic-polynomial tail]\label{lem:harmonic-tail}
Fix a time interval \(I\) and concentric balls
\[
B_r(x_0)\subset B_R(x_0),
\qquad
0<r\le R/4,
\]
whose \(\ell\)-neighborhoods remain in the solution domain. Let \(\zeta\in C_c^\infty(B_{2R}(x_0))\) satisfy \(\zeta\equiv1\) on \(B_R(x_0)\), and define for almost every \(t\)
\begin{equation}\label{eq:active-pressure}
P^a(\cdot,t)
:=
\mathcal R_i\mathcal R_j
\left[
\zeta\bigl(U_iU_j+R_{ij}\bigr)
\right](\cdot,t),
\qquad
P^h:=P-P^a.
\end{equation}
Then \(P^h(\cdot,t)\) is harmonic in \(B_R(x_0)\). For every integer \(m\ge0\), there is a harmonic polynomial \(H_m(\cdot,t)\) of degree at most \(m\) such that
\begin{equation}\label{eq:harmonic-tail}
\norm{P^h(\cdot,t)-H_m(\cdot,t)}{L^{3/2}(B_r)}
\le
C_m
\left(\frac rR\right)^{m+3}
\norm{P^h(\cdot,t)-c(t)}{L^{3/2}(B_R)}
\end{equation}
for every scalar function \(c(t)\).

If \(\phi\) is supported in \(B_r(x_0)\times I\) and \(\norm{\nabla\phi}{L^\infty}\le M_\nabla r^{-1}\), then the harmonic-tail contribution to pressure work satisfies
\begin{align}
\mathcal T_{m}^{h}[\phi]
&:=
 r^{-1}
\int_I\int_{B_r}
|P^h-H_m|\,|U|\,|\nabla\phi|\dxdt
\label{eq:T-h-def}
\\
&\le
C_mM_\nabla
\left(\frac rR\right)^{m+3}
 r^{-2}
\norm{U}{L^3(I\times B_r)}
\norm{P^h-c(t)}{L^{3/2}(I\times B_R)}.
\label{eq:T-h-bound}
\end{align}
\end{lemma}

\begin{proof}
Taking the divergence of \eqref{eq:coarse-momentum} gives
\begin{equation}\label{eq:pressure-poisson}
-\Delta P
=
\partial_i\partial_j(U_iU_j+R_{ij})
\end{equation}
in the interior. By the definition of the Riesz transforms,
\[
-\Delta P^a
=
\partial_i\partial_j
\left[
\zeta(U_iU_j+R_{ij})
\right].
\]
Since \(\zeta=1\) on \(B_R\), \eqref{eq:pressure-poisson} implies \(\Delta P^h=0\) there.

For a harmonic function \(h\) in \(B_R\), the interior \(L^p\)-derivative estimate with \(p=3/2\) gives
\begin{equation}\label{eq:harm-derivative}
\norm{\nabla^{m+1}h}{L^\infty(B_{R/2})}
\le
C_m R^{-(m+1)-2}
\norm{h-c}{L^{3/2}(B_R)}.
\end{equation}
Let \(H_m\) be the Taylor polynomial of \(h\) at \(x_0\) through degree \(m\). Because derivatives of a harmonic function are harmonic and the Taylor polynomial of a harmonic function is harmonic term by term, \(H_m\) is harmonic. Taylor's theorem and \eqref{eq:harm-derivative} yield, for \(x\in B_r\),
\[
|h(x)-H_m(x)|
\le
C_m r^{m+1}R^{-m-3}
\norm{h-c}{L^{3/2}(B_R)}.
\]
Multiplication by \(|B_r|^{2/3}\simeq r^2\) gives \eqref{eq:harmonic-tail}. Apply this at almost every time and then use H\"older's inequality in space-time together with \(|\nabla\phi|\le M_\nabla r^{-1}\) to obtain \eqref{eq:T-h-bound}.
\end{proof}

\begin{corollary}[Full pressure work versus active pressure work]\label{cor:pressure-split}
Under the assumptions of \Cref{lem:harmonic-tail},
\begin{align}
-r^{-1}\int_I\int PU\cdot\nabla\phi\dxdt
&=
-r^{-1}\int_I\int(P^a+H_m)U\cdot\nabla\phi\dxdt
+
\mathcal R_m^h[\phi],
\label{eq:pressure-work-split}
\\
|\mathcal R_m^h[\phi]|
&\le\mathcal T_m^h[\phi].
\label{eq:pressure-work-residual}
\end{align}
For \(m=0\), \(H_0(\cdot,t)\) is spatially constant and its work vanishes by \Cref{lem:gauge}. For \(m\ge1\), the low-order harmonic polynomial is a genuine pressure-work term and must be retained.
\end{corollary}

\begin{remark}
The local decomposition \eqref{eq:active-pressure} is a standard Calder\'on--Zygmund/harmonic decomposition. Local pressure formulations and the distinction between temporal gauges and spatially harmonic pressure are discussed.
\end{remark}

\section{A Conditional Detector-to-CKN Interface}\label{sec:conditional-interface}

We now record the precise way in which CKN resolution and work depletion fit together.  The point is not to assert an unconditional detector theorem, but to state the exact implication that would follow once a separate coarse observability estimate is available.

Let \(\Aact_k(G^\ell)\) be a finite-dimensional active coefficient norm of the form used in \Cref{sec:work-depletion}.  A coarse detector observability statement at scale \(r_k\) would have the form
\begin{equation}\label{eq:coarse-to-active-observability}
\Psi^\ell(z_0,r_k)\ge c_0\eps_0
\quad\Longrightarrow\quad
\Aact_k(G^\ell)\ge a_0,
\end{equation}
for constants \(c_0,a_0>0\) on a specified compact or admissible class of coarse-grained packages.  This implication is not a consequence of the resolution lemma or of the local energy identity.  It is a separate compactness-separation problem for the signed distribution \(G^\ell\).

\begin{theorem}[Resolution plus work-depletion template]\label{thm:resolution-work-template}
Fix a finite chain and a common coarse length \(\ell>0\) as in \Cref{sec:work-depletion}.  Let \(\eps_0>0\), let \(0<\eta<1/4\), and set \(c_0=1/4-\eta\).  Suppose that for every scale \(r_k\) under consideration the coarse observability implication \eqref{eq:coarse-to-active-observability} holds with the active norm used in \Cref{thm:main}.

If
\[
\Psi(z_0,r_k)\ge \eps_0
\qquad\text{and}\qquad
\Omega^\ell(z_0,r_k)\le \eta\eps_0,
\]
then the selected work weight \(\widehat\phi_k\) from \Cref{thm:main} satisfies
\begin{equation}\label{eq:bad-scale-work-alternative}
|\W_k[\widehat\phi_k]|\ge c_k a_0.
\end{equation}
Equivalently, every residual-small full CKN-bad scale lies in one of two branches:
\[
\W_k^+\ge c_k a_0
\qquad\text{or}\qquad
\W_k^-\ge c_k a_0.
\]
Moreover the selected work quantities obey the finite-chain inequality
\begin{equation}\label{eq:template-chain-bound}
\sum_{k=0}^{N-1} w_k(\W_k^+ + \D_k)
\le
\Ek_0^-+
\sum_{k=0}^{N-1}w_k|\Lk_k|+
\sum_{k=0}^{N-1}w_k\W_k^-.
\end{equation}
\end{theorem}

\begin{proof}
The resolution lemma, \Cref{cor:coarse-visibility}, gives
\[
\Psi^\ell(z_0,r_k)\ge \left(\frac14-\eta\right)\eps_0=c_0\eps_0
\]
whenever \(\Psi(z_0,r_k)\ge\eps_0\) and \(\Omega^\ell(z_0,r_k)\le\eta\eps_0\).  The assumed coarse observability implication \eqref{eq:coarse-to-active-observability} then gives \(\Aact_k(G^\ell)\ge a_0\).  The active extraction part of \Cref{thm:main} gives
\[
|\W_k[\widehat\phi_k]|\ge c_k\Aact_k(G^\ell)\ge c_k a_0,
\]
which is \eqref{eq:bad-scale-work-alternative}.  The branch alternative is just the decomposition into positive and negative parts.  Finally, \eqref{eq:template-chain-bound} is exactly the weighted telescoping estimate in \Cref{thm:main}.
\end{proof}

\begin{remark}[What remains open]
The assumption \eqref{eq:coarse-to-active-observability} is the hard missing bridge.  Large \(\Psi^\ell\) may still be invisible to a finite signed detector because of pressure--flux cancellation, poor alignment with the chosen active profiles, harmonic pressure tails, or coherent resolved fields with little localized work.  The theorem above is therefore best read as a clean interface statement: once coarse observability is supplied, the rest of the finite-chain ledger is already available from the exact resolved energy identity.
\end{remark}

\section{Diagnostic Examples and Scaling Ledger}\label{sec:diagnostics}

\subsection{A spatially constant velocity with harmonic pressure}

\begin{example}[Harmonic pressure performs real local work]\label{ex:harmonic-pressure}
Let \(a\in C^1(I;\R^3)\), and set
\[
U(x,t)=a(t),
\qquad
P(x,t)=-a'(t)\cdot x,
\qquad
R=0.
\]
Then
\[
\partial_tU+\nabla P=0,
\qquad
\nabla\cdot U=0,
\]
so this is a smooth local Navier--Stokes solution. The pressure is spatially harmonic and \(\Pi=0\), but for a spatial cutoff \(\chi\),
\[
\int_I\int\chi G\dxdt
=
\int_I\int\chi\,\nabla\cdot(PU)\dxdt
=
-\int_I a(t)\cdot a'(t)\dt\int\chi\dx.
\]
This is exactly the localized kinetic-energy change. Thus harmonic pressure cannot be removed as a gauge.
\end{example}

\subsection{Forward transfer and backscatter are both algebraically possible}

Even when \(R\ge0\), the sign of \(\Pi=-R:\nabla U\) is not fixed. For example, with
\[
R=e_1\otimes e_1,
\]
a divergence-free linear gradient \(\nabla U=\operatorname{diag}(-1,1,0)\) gives \(\Pi=1\), while \(\nabla U=\operatorname{diag}(1,-1,0)\) gives \(\Pi=-1\). These are algebraic sign tests for the coarse equation; they are not assertions that an arbitrary prescribed pair \((U,R)\) is generated by a Navier--Stokes solution. Their role is to show why the theorem must contain a backscatter alternative.

\subsection{Pressure--flux cancellation}

If the two channels in \eqref{eq:channels} have opposite signs, their magnitudes may be large while \(\W\) is small. The quantity \(\Cpf\) in \eqref{eq:cancellation} records exactly this loss. The active detector in this paper is a detector for the combined distribution \(G\), not for the sum of absolute values of its pressure and flux components. Therefore no estimate of the form
\[
\int|\Pi|+\int|P|^{3/2}
\lesssim
\text{endpoint kinetic-energy drop}
\]
is claimed.

\subsection{Scale-invariant normalization ledger}

\begin{longtable}{>{\raggedright\arraybackslash}p{0.25\textwidth}
                  >{\raggedright\arraybackslash}p{0.47\textwidth}
                  >{\centering\arraybackslash}p{0.12\textwidth}}
\toprule
Quantity & Definition & Scaling degree \\
\midrule
\endfirsthead
\toprule
Quantity & Definition & Scaling degree \\
\midrule
\endhead
Endpoint energy
& \(r^{-1}\int\frac12|U|^2\chi\dx\)
& \(0\) \\
Resolved dissipation
& \(r^{-1}\int\phi|\nabla U|^2\dxdt\)
& \(0\) \\
Combined work
& \(r^{-1}\int(\phi\Pi-PU\cdot\nabla\phi)\dxdt\)
& \(0\) \\
Temporal/spatial energy leakage
& \(r^{-1}\int\frac12|U|^2(\partial_t\phi+\Delta\phi)\dxdt\)
& \(0\) \\
Convective leakage
& \(r^{-1}\int\frac12|U|^2U\cdot\nabla\phi\dxdt\)
& \(0\) \\
Stress-transport leakage
& \(r^{-1}\int(RU)\cdot\nabla\phi\dxdt\)
& \(0\) \\
Active coefficient
& \(r^{-1}\pair{G}{\chi\psi^\sharp}\)
& \(0\) \\
Harmonic-tail residual
& right side of \eqref{eq:T-h-bound}
& \(0\) \\
\bottomrule
\end{longtable}

\section{Conclusion}\label{sec:conclusion}

We have proved a finite-scale coarse-grained obstruction calculus for the local Navier--Stokes regularity problem.  The first component is the unconditional CKN resolution estimate
\[
\Psi(r)\le4\Psi^\ell(r)+4\Omega^\ell(r),
\]
which separates full local badness into a resolved velocity--pressure contribution and an explicit subfilter residual.  The second component is an exact finite-chain depletion theorem for the combined pressure--flux work distribution
\[
G^\ell=\Pi^\ell+\nabla\cdot(P^\ell U^\ell).
\]
It shows that finite-dimensional detected resolved work obeys a localized energy payment law with explicit leakage and backscatter.

\end{document}